\documentclass[11pt]{article}

\usepackage{fullpage}
\usepackage{amsmath, amsthm}
\usepackage{amsfonts,amssymb,psfrag}
\usepackage[utf8]{inputenc}
\usepackage{hyperref,qsymbols,color}
\usepackage{graphicx, subfigure}
\usepackage{bbm}
\usepackage{enumerate}
\def \ben{\begin{eqnarray}}
\def \een{\end{eqnarray}}
\def \be{\begin{eqnarray*}}
\def \ee{\end{eqnarray*}}
\def \beq{\begin{equation}}
\def \eq{\end{equation}}

\def \eref#1{(\ref{#1})}

\def \S{{\sf S}}
\def \RS{{\sf RS}}

\def \In{{\sf in}}
\def \comp{{\sf comp}}
\def \Stable{{\sf Stable}}
\def \s{{\sf Sto}}
\def \D{{\sf{Det}}}

\def \1{\mathbbm{1}}
\def \bT{\overline{T}}

\newcommand{\Z}{\mathbb{Z}_+}

\newcommand{\V}{\mathcal{V}(G)}
 \newqsymbol{`E}{\mathcal{E}}
 \newqsymbol{`P}{\mathcal{P}} 
\newtheorem{thm}{Theorem}[section]
\newtheorem{pro}[thm]{Proposition}
\newtheorem{cor}[thm]{Corollary}
\newtheorem{lem}[thm]{Lemma}

\newtheorem{de}[thm]{Definition}%[section]
\def \captionn#1{\begin{center}\begin{minipage}{14cm}\sf\caption{\small #1}\end{minipage}\end{center}}

\renewcommand{\baselinestretch}{1.2}

\title{A natural stochastic extension of the sandpile model on a graph\thanks{This work has been partially supported by ANR-08-BLAN-0190-04 A3.}}
\author{Yao-ban Chan \\ Fakult\"at f\"ur Mathematik\\ Universit\"at Wien \and  Jean-Fran\c cois Marckert \\ LaBRI, CNRS\\ Université Bordeaux 1 \and Thomas Selig \\ LaBRI, CNRS \\  Université Bordeaux 1 }
\begin{document}

\maketitle

\begin{abstract}
We introduce a new model of a stochastic sandpile on a graph $G$ containing a sink. When unstable, a site sends one grain to each of its neighbours independently with probability $p \in (0,1]$. For $p=1$, this coincides with the standard Abelian sandpile model. In general, for $p\in(0,1)$, the set of recurrent configurations of this sandpile model is different from that of the Abelian sandpile model. We give a characterisation of this set in terms of orientations of the graph $G$. We also define the lacking polynomial $L_G$ as the generating function counting this set according to the number of grains, and show that this polynomial satisfies a recurrence which resembles that of the Tutte polynomial.
\end{abstract}

\section{Introduction}

In this paper, we analyse a stochastic generalisation of the Abelian sandpile model (ASM). Informally (we provide a formal definition later), the ASM operates on a graph where each vertex has a number of `grains of sand' on it. At every unit of time, another grain of sand is added at a random vertex $v$. If this causes the number of grains at $v$ to exceed its degree, $v$ topples, sending one grain to each of its neighbours. This may cause other vertices to topple, and we continue until the configuration is stable, i.e. no vertex can topple anymore. A special vertex, the sink, can absorb any number of grains and never topples. It is possible to show that eventually this model will be trapped in a set of configurations, called recurrent configurations.

This model arose from work by Bak, Tang and Wiesenfeld \cite{BTW1, BTW2} and was named and formalised by Dhar \cite{DH1}. It displays a phenomenon known as self-organised criticality \cite{Turcotte, Jensen}, in which characteristic length or time scales break down in the `critical' steady state. When this happens, the correlation between the number of grains at two vertices obeys a power-law decay, as opposed to an exponential decay often found in models away from criticality. Likewise, the average number of topplings that result from a single grain addition also obeys a power-law distribution. In this sense, the model is `non-local', as grains added at a vertex may have an effect on vertices that are far away. Physically, this model (and self-organised criticality in general) has been used in applications as wide as forest fires \cite{Ricotta}, earthquakes \cite{Bak-earthquake} and sediment deposits \cite{Rothman}.

Mathematically, the ASM has been heavily studied, and we shall not list out all the references here. We refer interested readers to Dhar's papers \cite{DH1,Dhar} and to the excellent review on the subject by Redig \cite{Red}. Some relevant results to our current work discussed in \cite{DH1,Red} include:
\begin{itemize}
\item The number of recurrent configurations is equal to the number of spanning trees of the graph.
\item There is an algorithm (called the \emph{burning algorithm}) which determines if a given configuration is recurrent or not. This finds, or establishes the non-existence of, a subgraph not including the sink on which the configuration is stable. This algorithm constructively establishes a bijection between recurrent configurations and spanning trees.
\item In the steady state of the model, each recurrent configuration is equally likely.
\end{itemize}

In the ASM, the only randomness occurs in the vertices that we add grains to. We introduce a variation on this model, where the topplings themselves are also random. More precisely, we fix a probability $p \in (0,1]$, and when a site is unstable, each neighbour independently has a probability $p$ of receiving a grain from the unstable site. In this way, an unstable site may remain unstable after toppling but, as in the original model, the process continues until the configuration is stable. If $p=1$, this is identical to the ASM.

Although this new model appears similar to the ASM, a closer inspection reveals some qualitative differences. In particular, the model will again become trapped in a set of recurrent configurations, but this set is not equal to the set of recurrent configurations in the ASM. Furthermore, each recurrent configuration is not equally likely, and the steady state measure now depends on $p$. The aim of this paper is to study the behaviour of this new model, particularly in these respects. We prove a characterisation of the recurrent configurations in terms of orientations of the graph edges, and provide a ``Tutte polynomial-like" formula which counts these configurations in terms of their numbers of grains of sand.  

The stochastic sandpile model (SSM) we propose is an appropriate generalisation of the ASM, due to the aforementioned relation between the Tutte polynomial and a counting of the recurrent configurations. A similar result has been proved by L\'opez \cite{Lop} for the ASM; in this model, the lacking polynomial we define later has an expression using the ``standard'' Tutte polynomial\footnote{Indeed, one can pass from their ``level polynomial'' to the lacking polynomial by a change of variables and the multiplication by a monomial factor, since the level of a configuration is $\sum \eta_v -|E|$ if the number of lacking particles is $2|E|-1-\sum \eta_v$.}. See also Cori and Le Borgne \cite{CL} and Bernardi \cite{OB} for combinatorial explanations of this fact. Although several instances of random sandpile models have already been introduced in the literature (see for example Dhar \cite[Section 6]{Dhar}, Kloster \emph{et al.} \cite{KMT}, Manna \cite{Manna}), as far as we are aware none of these models have any known link with Tutte-like polynomials. Variations on the basic sandpile model also include an asymmetric version \cite{Speer}, which devised an analogous algorithm to the burning algorithm called the \emph{script algorithm}.

In Section \ref{sec:model}, we formally define the ASM and introduce our stochastic generalisation of it (the SSM). In Section \ref{sec:result}, we present our results on the SSM, giving detailed proofs in Section \ref{sec:proof}. Finally, we offer a brief conclusion in Section \ref{sec:conclusion}.

\section{The model}
\label{sec:model}

In this section, we formally define our model and its associated notation. We start with a brief review of the established Abelian sandpile model.

\subsection{The ASM}

We first define the class of graphs $G=(V \cup \{s\},E)$ that underly the model. $G$ must be finite, unoriented, connected and loop-free. It may have multiple edges, and it contains a distinguished vertex $s$ that we call the \emph{sink}.   
The set of these graphs is denoted ${\cal G}$. We use the notation $u\sim v$ to denote that $u$ and $v$ are adjacent in $G$, i.e. $\{u,v\}\in E$. 

A sandpile \textit{configuration} on $G$ is a vector $\eta=(\eta_v, v \in V) \in \Z^{ \vert V \vert}$. The number $\eta_v$ represents the number of grains of sand present at the vertex $v$ in the configuration $\eta$. When this number exceeds a certain threshold, the vertex is said to be \emph{unstable}, and will \textit{topple}, sending one grain of sand to each of its neighbours. Typically and throughout this paper, the threshold is set to the degree of that vertex. The sink plays a special role in that it can absorb any number of grains, and as such never topples. 

Two different configurations play an important role in this paper. The first one $\eta^{\max}_G$ (or $\eta^{\max}$ when $G$ is clear from the context) is the configuration 
\beq
\eta^{\max}_G:=(d^G(v), v \in V),
\eq
where $d^G(v)$ is the degree of the vertex $v$ in $G$.  The second one is $\1_a$ for some vertex $a$ in $V$, 
\beq
\1_a:=(\delta_{a,v}, v \in V),
\eq 
where $\delta_{a,v}$ is the Kronecker symbol, meaning that all co-ordinates of $\1_a$ are 0 except for at position $a$, where it is 1. 
\begin{de} 
A configuration $\eta = (\eta_v, v \in V)$ is called \emph{stable} if  $\eta_v \leq d^G(v)$ for all $v \in V$. We write $\Stable(G)$ for the set of all stable configurations on $G$.
\end{de}
\noindent $\eta^{\max}_G$ is clearly the maximum stable configuration in terms of the total number of grains in the configuration.

Now define the toppling operator $T_x$ corresponding to a toppling at $x\in V$ by
\beq\label{eq:full-toppling}
T_x(\eta)=\eta - d^G(x)\1_x+\sum_{y\sim x, y\neq s} \1_y,
\eq
where configurations are added site-by-site. A toppling $T_x$ is called \textit{legal} if $\eta_x > d^G(x)$.

We now define for any configuration $\eta$ its \emph{stabilisation} $\S(\eta)$ as follows. If $\eta$ is stable, then $\S(\eta)=\eta$; otherwise
\beq
\S(\eta) = T_{x_m}(...(T_{x_1} (\eta))...), 
\eq
with the requirements that $T_{x_1},...,T_{x_m}$ is a sequence of legal topplings, and that the configuration $\S(\eta)$ is stable. The fact that this operator is well-defined, i.e. the stabilisation of a configuration is independent of the order of the topplings that is used to stabilise it, is not immediately obvious. We refer the interested reader to Proposition 3.7 (and Lemma 3.6) in Redig \cite{Red}.

\subsubsection*{Markov chain structure of the ASM}

The ASM has a Markov chain structure which is defined as follows. Assume that there are defined some i.i.d. random variables $(X_i,i\geq 1)$ taking their values in $V$ according to a distribution $\mu$ (where the support of $\mu$ is $V$), defined on a common probability space $(\Omega,{\cal A},`P)$. 
The sandpile process starts from any stable configuration $\eta_0$. We define a Markov chain $(\eta_i,i\geq 0)$ with values in $\Stable(G)$. Given $\eta_{i-1}$ for any $i\geq 1$, $\eta_i$ is obtained as follows:
\begin{itemize}
\item Add a grain at position $X_i$ to the configuration $\eta_{i-1}$. Let $\eta'_i$ be the obtained configuration.
\item Let $\eta_i$ be the stabilisation of $\eta_i'$, that is $\eta_i=\S(\eta'_i)$  (in some cases no toppling is needed). 
\end{itemize}
Since a grain can be added anywhere with positive probability, it is immediately seen that from any configuration $\eta$, the maximal configuration $\eta^{max}$ can be reached with positive probability. It follows that the set of recurrent configurations for the Markov chain is the unique recurrent class containing the maximal configuration. We denote this set by $\D(G)$, and as previously noted, it is in bijection with the set of spanning trees of $G$ (see Redig \cite{Red}).

\subsection{The SSM}
\label{sec:SSM}
In this model, we make the topplings random. A probability $p \in (0,1]$ is fixed, and when a site is unstable, each neighbour independently has a probability $p$ of receiving a grain from the unstable site.  In this way, an unstable site may remain unstable after toppling, but as in the original model, the process continues until the configuration is stable. Different topplings are done independently. If $p=1$, this reduces to the ASM, so we assume this is not the case.

There are some slight difficulties in establishing that this model is indeed well-defined. Two of them are taken into account in the proof of Theorem \ref{thm:well-defi} below. A third one arises by noticing that topplings commute in the ASM. However, in the SSM, when two legal topplings are possible, it is not clear that they commute, since we only talk about what to do ``in probability". To solve this, we describe a probability space where the laws of the topplings are indeed those discussed at the beginning of this section and for which the commutation of legal topplings takes place as needed.

We now give a formal definition of the SSM. The definition of terms introduced in relation to the ASM, apart from the toppling operator, are unchanged. All the random variables discussed below are defined on a common probability space $(\Omega,{\cal A},`P)$, and they are all independent.
\begin{itemize}
\item The variables $(X_i,i\geq 0)$ are distributed according to a distribution $\mu$ with support equal to $V$. As before, they represent a sequence of arrival places of grains.
\item For any $x\in V, e \in E$ such that $e = \{x,y\}$ and any $i\geq 0$, $B_i(x,e)$ is a Bernoulli random variable with parameter $p$. It represents the number of particles going from $x$ to $y$ along $e$ due to the $i$th toppling of $x$.
\end{itemize}

We now define $\bT^{(i)}_x$ as the $i$th toppling at $x\in V$
\beq
\bT^{(i)}_x(\eta)=\eta- \sum_{e \ni x}B_i(x,e) \1_{x} +\sum_{y \neq s, e=\{x,y\} \in E}B_i(x,e)\1_{y}.
\eq
We will call any realisation of such a toppling a \textit{stochastic toppling}. Note that with positive probability, the $i$th toppling at $v$ is a \it full toppling, \rm meaning that $B_i(x,e)=1, \forall e \ni x$, in which case $\bT^{(i)}_x$ coincides with the ASM toppling $T_x$.

On the probability space $(\Omega,{\cal A},\mathcal{P})$, the SSM is well defined when a starting configuration is specified. 
We then use the same definition of a legal toppling, and define the stabilisation of a configuration $\eta$ by
\beq
\RS(\eta) = \bT_{x_n}(...(\bT_{x_1} (\eta))...), 
\eq
where we take $\bT_{x_j}=\bT^{(i)}_{x_j}$ if the vertex $x_j$ appears exactly $i-1$ times in the sequence of topplings $\bT_{x_1},...,\bT_{x_{j-1}}$, and with the same requirements as for the ASM, i.e. that $\bT_{x_1},...,\bT_{x_m}$ is a legal sequence of topplings and $\RS(\eta)$ is stable. The notation $\RS$ stands for the ``random stabilisation'' we define.

\begin{thm}\label{thm:well-defi} For any graph $G$ in ${\cal G}$, the stabilisation operator $\RS$ is almost surely well-defined.
\end{thm}
\begin{proof} Two things have to be shown. The first one is that the stabilisation process eventually ends. Here, as opposed to the ASM, a loop may appear in the toppling process. For example, if $a \sim b$ and $a$ and $b$ are both unstable, and for all $i$ we have $\bT_a^{(i)} = \1_b$ and $\bT_b^{(i)} = \1_a$, then a repeated loop occurs where $a$ only sends one grain to $b$ and $b$ only sends one grain to $a$, \emph{ad infinitum}. This example is obviously contrived, and it is easy to see that this occurs with probability 0, but some more complex loops could occur. 

To show that this almost surely does not happen, we note that in the ASM, stabilisation takes a finite number of topplings regardless of the initial configuration. Let $N_v$ be the maximum number of full topplings at $v$ needed (by the ASM) to stabilise any configuration arising from the addition of 1 grain to a stable configuration; let $N=\max_{v\in V} N_v$. Since there are a finite number of such configurations, $N$ is finite. Now for any $i$ and $x$, $\bT^{(i)}_x$ is the ``full toppling'' $T_x$ with positive probability. A simple renewal argument shows that in an infinite sequence of topplings, the number of full topplings done at every vertex is almost surely larger than $N$. From here the conclusion follows (note that this doesn't mean that full topplings are needed to stabilise a configuration).

The second thing we have to prove is that that the stabilisation operator is well-defined; that is, starting from an unstable configuration $\eta$, changing the order of the site stabilisations does not change the final result $\RS(\eta)$. This is the case on $(\Omega,{\cal A},`P)$ since whatever is done elsewhere, the $i$th toppling at $v$ is $\bT^{(i)}_v$. We can then use the same argument as in Lemma 3.6 and Proposition 3.7 of Redig \cite{Red}, with a simple adjustment to take into account the fact that $i\mapsto \bT^{(i)}_x$ is not constant. 
\end{proof}

\subsubsection*{Markov chain structure of the SSM}

The SSM has a Markov chain structure which is analogous to the ASM. The sandpile process starts from any stable configuration $\eta_0$. We define a Markov chain $(\eta_i,i\geq 0)$ with values in $\Stable(G)$. Given $\eta_{i-1}$ for any $i\geq 1$, we obtain $\eta_i$ as follows:
\begin{itemize}
\item Add a grain at position $X_i$ to the configuration $\eta_{i-1}$. Let $\eta'_i$ be the resulting configuration.
\item Let $\eta_i$ be the stabilisation of $\eta_i'$, that is $\eta_i=\RS(\eta'_i)$  (in some cases no toppling is needed). 
\end{itemize}

Note that $\RS$ is not defined \emph{ex nihilo} as in the ASM. On $(\Omega,{\cal A},`P)$, when $\RS$ is applied, its action depends on all the previous topplings taken to reach the current state. In view of this, it is more proper to write $\RS_i$ or even $\RS_i(\eta_0)$ instead of $\RS$. However, for the sake of brevity we write $\RS$ instead. 

Once again, since the support of $\mu$ is $V$, the maximal configuration $\eta^{max}$ is recurrent, and the set of recurrent configurations for the Markov chain is the unique recurrent class containing $\eta^{max}$. We denote this set by $\s(G)$, and call any element of it \textit{stochastically recurrent} (or \emph{SR}). We call the configurations which are recurrent under the ASM \emph{deterministically recurrent} or \emph{DR} in contrast.

\begin{pro}\label{DS} For any $G\in {\cal G}$,
\[\D(G) \subseteq \s(G).\]
\end{pro}
\begin{proof}
This is easily seen because $\S$ always uses a finite number of topplings. When constructing a stochastic stabilisation, the deterministic stabilisation must have a positive probability of being reconstructed using only full topplings.
\end{proof}

The two sets $\D(G)$ and $\s(G)$ are not equal in general. A simple counter-example is shown in Figure \ref{fig:triangle}. For this graph, the configuration which has $2$ grains at each vertex is SR. It can be reached from the DR configuration $[3,1,2]$ (with 3 grains at the vertex $v_1$ leading to the sink), by adding a grain to $v_1$ and toppling it, sending one grain to the sink and one to $v_2$. However, it is not DR as it fails the burning algorithm test --- the graph with the sink removed is a forbidden subconfiguration.

\begin{figure}[h]
\psfrag{s}{$s$}
\psfrag{1}{$v_1$}
\psfrag{2}{$v_2$}
\psfrag{3}{$v_3$}
\psfrag{4}{$v_4$}
\psfrag{5}{$v_5$}
\centerline{\includegraphics{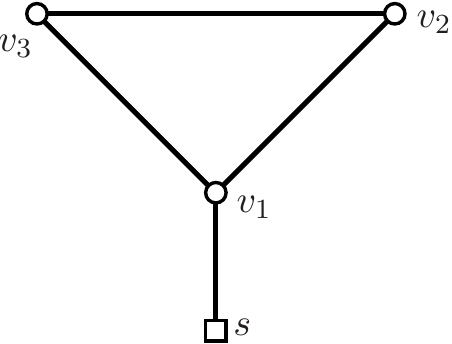}}
\captionn{\label{fig:triangle}A graph for which $\D(G)$ is a strict subset of $\s(G)$. In this case $\D(G)=\{[3, 1, 2], [3, 2, 1], [3, 2, 2]\}$, and $\s(G)=\{[2, 2, 2]\} \cup \D(G)$.}
\end{figure}

To determine $\s(G)$, it is clear that $\eta$ is SR if and only if there exists a finite sequence of adding of grains and topplings such that $\eta$ is reached from $\eta^{max}$ through this sequence. In the rest of the paper, we prove a more useful characterisation of the stochastically recurrent states.

\section{Main results}\label{sec:result}

In this section, we state the main results of this paper. Our two main results are a characterisation of the stochastically recurrent states in terms of graph orientations and a recurrence for the lacking polynomial (which we define below). Proofs of these results will be given in Section \ref{sec:proof}.

\subsection{Graph orientations}

Our first result characterises the SR states in terms of graph orientations. Take a graph $G=(V \cup \lbrace s \rbrace, E) \in {\cal G}$. We define an orientation on $G$ to be an orientation of each edge of $E$ (when $G$ has multiple edges, all of them are oriented independently). We write $(a,b)$ or $a\to b$ to denote that the edge $\{a,b\}$ is oriented from $a$ to $b$.

\begin{de} \label{comp_or}
Let $G=(V \cup \lbrace s \rbrace, E) \in {\cal G}$. Take a sandpile configuration $\eta$ on $G$. We define the \emph{lacking number} of $\eta$ at $v$ as the number of grains at $v$ less than its maximum value:
\[l^G_\eta(v) = d^G(v) - \eta_v.\]
Now let $O$ be an orientation on $G$ and let $\In^G_O(v)$ be the number of incoming edges to $v$ in $O$. We say that $\eta$ is \emph{compatible} with $O$ (and likewise $O$ is compatible with $\eta$) if $\forall v \in V$,
\beq\label{comp_cond}
\In^G_O(v) \geq 1 + l^G_{\eta}(v).
\eq
We denote the set of stable configurations that are compatible with $O$ as $\comp(O)$.
\end{de}
In situations where it is clear, we will omit the superscript $G$ for brevity.

Note that there may be several configurations compatible with a particular orientation. Likewise, there may be several orientations compatible with any given configuration. For instance, the maximal configuration $\eta^{max}$ is compatible with any orientation where each vertex has at least one incoming edge. 

\begin{thm} \label{car_sr}
Let $G=(V \cup \lbrace s \rbrace, E) \in {\cal G}$. Then a (stable) configuration $\eta$ is stochastically recurrent if and only if there exists an  orientation $O$ on $G$ such that $\eta \in \comp(O)$. In other words,
\beq
\s(G) = \bigcup_O \comp(O),
\eq
where the union is taken over all orientations on $G$.
\end{thm}

Furthermore, there is also a way of characterising DR configurations using orientations.

\begin{thm}\label{th:car_dr}
A (stable) configuration $\eta$ is deterministically recurrent if and only if there exists an orientation $O$ of $G$ with no directed cycles such that $\eta\in \comp(O)$.
\end{thm}

This theorem is intuitive given the bijection between DR configurations and spanning trees, as any spanning tree can induce a (not necessarily unique) orientation with no directed cycles.

\subsection{The lacking polynomial}

\begin{figure}[ht]
\psfrag{s}{$s$}
\psfrag{1}{$v_1$}
\psfrag{2}{$v_2$}
\psfrag{3}{$v_3$}
\psfrag{4}{$v_4$}
\psfrag{5}{$v_5$}
\centerline{\includegraphics[height=3.5cm]{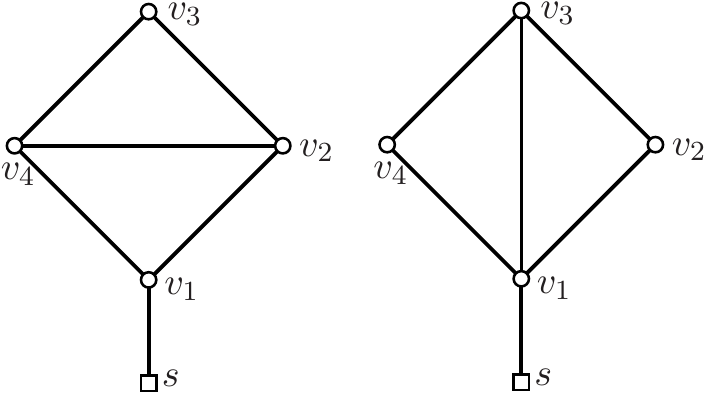}}
\captionn{\label{fig:exa1} The graph $G$ on the left has $\D(G)=\{[3, 1, 2, 3], [3, 2, 1, 3], [3, 2, 2, 3], [3, 3, 1, 2], [3, 3, 1, 3], [3, 3, 2, 1], [3, 3, 2, 2], [3, 3, 2, 3]\}$ and $\s(G)=\D(G)\cup \{[1, 3, 2, 3], [2, 2, 2, 3], [2, 3, 1, 3], [2, 3, 2, 2], [2, 3, 2, 3], [3, 2, 2, 2]\}$. Thus the lacking polynomial is $L_G(x)=1+4x+9x^2$. The graph on the right has lacking polynomial $1+4x+8x^2$. The DR configurations are equinumerous in both graphs since they are in bijection with the sets of spanning trees, which can be identified with each other. This does not hold for the SR configurations.}
\end{figure}
Our second result uses Theorem \ref{car_sr} in order to classify SR configurations according to the total number of grains, or equivalently to the number of grains removed from the maximal configuration. We do this by means of the lacking polynomial, which we now define.

\begin{de}
Let $G=(V \cup \lbrace s \rbrace, E) \in {\cal G}$. The \emph{lacking polynomial} $L_G$ of $G$ is the generating function of the stochastically recurrent configurations on $G$, with $x$ conjugate to the number of lacking particles in the configuration:
\begin{equation}\label{eq:L}
L_G(x) = \sum_{\eta \in \s(G)} x^{\ell(\eta)},  
\end{equation}
where
\[\ell(\eta) = \sum_{v\in V} l^G_v(\eta).\]
\end{de}
An example of the lacking polynomial is shown in Figure \ref{fig:exa1}. Note that we can use the lacking polynomial to count the number of SR configurations, as $|\s(G)| = L_G(1)$.

Before stating the main result of this section (Theorem \ref{thm:LP}), which gives a Tutte-like formula to compute $L_G$, we state (without proof) some propositions concerning some special graphs which serve as an initialisation for the computations (since $L_G$ will be expressed in terms of the lacking polynomials of some graphs smaller than $G$). Related illustrations can be found in Figure \ref{fig:illus}.
\begin{pro}\label{pro:init_cond}
Let $G = (V \cup \lbrace s \rbrace,E) \in {\cal G}$.
\begin{enumerate}
\item If $V=\{u\}$ and there are $k$ edges between $u$ and $s$, then 
\[L_G(x)= \sum_{i=0}^{k-1}x^{i}.\]
\item If $G$ is a tree, then $L_G(x)=1$.
\item If we can write $G$ as the union of connected graphs $G_i = (V_i\cup \lbrace s \rbrace, E_i), i = 1, \ldots, k$, so that the $V_i$ are mutually disjoint, then
\[L_G(x) = \prod_{i=1}^k L_{G_i}(x).\]
We say that $G$ is the \emph{product} of the $G_i$.
\end{enumerate}
\end{pro}

We now state that the pruning of ``tree branches" of a graph $G$ does not change its lacking polynomial.
\begin{de}\label{def_TB}
Let $G = (V \cup \lbrace s \rbrace,E) \in {\cal G}$. A \emph{tree branch} of $G$ is a subgraph  $T = (V' \cup \lbrace r \rbrace, E')$ of $G$ which is a tree attached to the rest of $G$ at the vertex $r$. In other words, $T$ is a tree, $d^T(v) = d^G(v)$ for any  $v \in V'$, and $r\in V\setminus V'$.
\end{de}

\begin{lem}\label{pro_TB}
Let $G = (V \cup \lbrace s \rbrace,E) \in {\cal G}$, and let $T = (V' \cup \lbrace r \rbrace, E')$ be a tree branch of $G$. Define $G \setminus T := ((V \setminus V') \cup \lbrace s \rbrace, E \setminus E')$ as $G$ with the tree branch removed. Then
\[ L_{G \setminus T}(x) = L_G(x). \]
\end{lem}

\begin{figure}[ht]
\psfrag{s}{$s$}
\centerline{\includegraphics[height = 4.5cm]{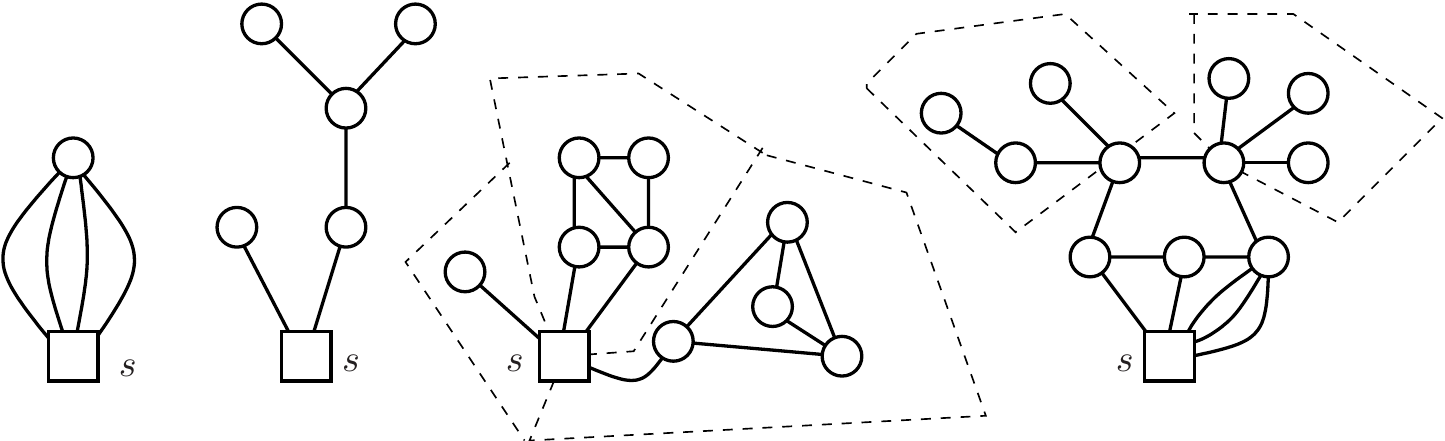}}
\captionn{\label{fig:illus} From Proposition \ref{pro:init_cond}, the first graph has lacking polynomial $1+x+x^2+x^3$, the second graph (which is a tree) 1, and the lacking polynomial of the third graph is the product of the lacking polynomials of the 3 graphs surrounded by dashed lines (all containing a sink at the same place).  Lemma \ref{pro_TB} says that in the fourth figure, the lacking polynomial is not affected by the removal of the two tree branches surrounded by dashed lines.}
\end{figure}

Note that this lemma implies the second statement of Proposition \ref{pro:init_cond}. We now give a formula allowing one to compute the lacking polynomial for a given graph; this formula is similar to the one used to compute the Tutte polynomial. First we require some definitions of edge deletion and contraction, similar to those used in the Tutte polynomial relation (see e.g. Bernardi \cite{OB} and references therein).

\begin{de}
Let $G=(V \cup \lbrace s \rbrace, E) \in {\cal G}$, and consider an edge $e = \{x,y\} \in E$, with $x,y \neq s$.
\begin{enumerate}
\item \textbf{Edge deletion}. The graph $G \setminus e$ is the graph $G$ with $e$ removed, i.e. $G \setminus e = (V \cup \{s\},E \setminus \lbrace e \rbrace)$.
\item \textbf{General edge contraction}. Define the graph $G.e$ as follows:
\begin{itemize}
\item If $e$ is simple, then $G.e$ is $G$ with $e$ contracted, i.e. $G.e = (V \cup \{x.y, s\} \setminus \lbrace x,y \rbrace, E \setminus \lbrace e \rbrace)$, where edges adjacent in $G$ to either $x$ or $y$ are now connected to $x.y$ instead.
\item If $e$ has multiplicity $k \geq 2$, contract one of these edges as above, and replace the other $k-1$ edges with $k-1$ edges $\{x.y,s\}$.
\end{itemize}
\end{enumerate}
\end{de}
\begin{figure}\centering
\psfrag{s}{$s$}\psfrag{a}{$a$}\psfrag{b}{$b$}\psfrag{e}{$e$}\psfrag{a.b}{$a.b$}
\includegraphics[height=4cm]{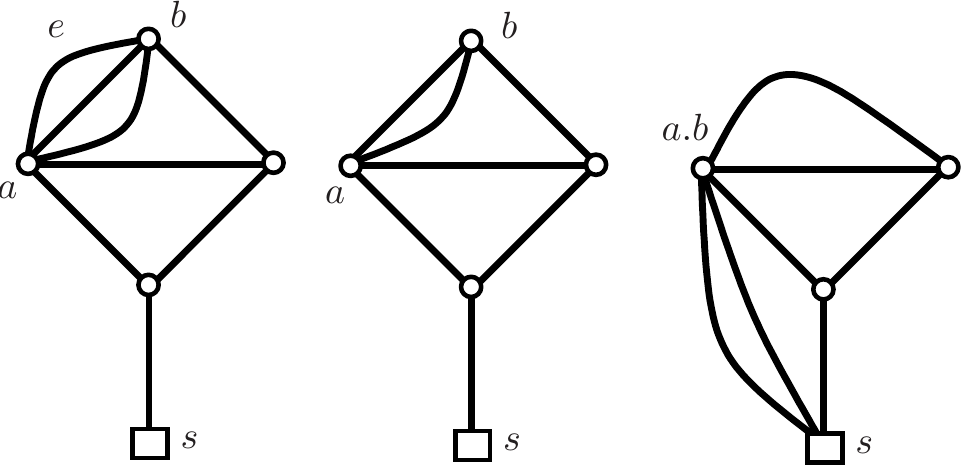}
\captionn{An example of the deletion and contraction operations. $G$ with the edge $e$ marked is on the left; $G \setminus e$ and $G.e$ are in the centre and right respectively.\label{fig:edge_manip}}
\end{figure}

We illustrate these operations in Figure \ref{fig:edge_manip}.

\begin{thm}\label{thm:LP}
Let $G = (V \cup \lbrace s \rbrace,E) \in {\cal G}$, and let $e$ be an edge of $E$ which is neither a bridge (i.e. removing $e$ doesn't disconnect the graph), nor connected to the sink. Then
\beq\label{eq:master-formula}
L_G(x) = x L_{G\setminus e}(x)+L_{G.e}(x).
\eq
\end{thm}
By way of contrast, the corresponding relation for the Tutte polynomial, where $e$ is not a loop or bridge, is
\[T_G(x,y) = T_{G\setminus e}(x,y) + T_{G.e}(x,y),\]
although here loops are allowed in $G$ and so $G.e$ denotes regular edge contraction rather than the version defined above.

One can check that the use of 
Theorem \ref{thm:LP}, Proposition \ref{pro:init_cond}, and Lemma \ref{pro_TB}
allows one to compute $L_G$ for any graph $G$ in ${\cal G}$ without referring to the SSM. More specifically, Theorem \ref{thm:LP} expresses $L_G$ in terms of graphs with one less edge. We can continue to use this theorem, and case 3 in Proposition \ref{pro:init_cond}, until we express $L_G$ in terms of graphs which only contain edges to the sink and bridges. The bridges must then form tree branches which are removed by Lemma \ref{pro_TB}, and case 1 in Proposition \ref{pro:init_cond} provides the lacking polynomials of the remainders. It is not immediately obvious that the polynomial so obtained does not depend on the edges that we choose to delete/contract; however, since the lacking polynomial itself is well-defined, this follows from Theorem \ref{thm:LP}.

A simple corollary of Theorem \ref{thm:LP} gives us the degree of the lacking polynomial.

\begin{cor}
Let $G = (V \cup \{s\},E) \in {\cal G}$. Let the \emph{number of cycles} $c(G)$ be the number of edges it is necessary to remove in order to turn $G$ into a tree, i.e. $c(G) = \vert E \vert - \vert V \vert$. Then the degree of the lacking polynomial $L_G$ is the number of cycles $c(G)$.\label{cor:degree}
\end{cor}
\begin{proof}
We use induction on $ \vert E \vert$. If $\vert E \vert = 1$ then $G=(\{u,s\},\{\{u,s\}\})$. Thus $L_G = 1$, $c(G)= 0$ as desired. Now consider a graph $G= (V \cup \{s\},E)$, and an edge $e$ which is not a bridge or connected to the sink. If there is no such edge, then after removing tree branches, which does not affect $c(G)$, we can express $G$ as the product of graphs of the form described in the first case of Proposition \ref{pro:init_cond}, which obviously satisfies the corollary. Otherwise, the graphs $G \setminus e$ and $G.e$ both have $ \vert E \vert - 1$ edges. Moreover, $c(G \setminus e) = c(G) - 1$ and $c(G.e) = c(G)$. The result then immediately follows by induction using \eref{eq:master-formula}.
\end{proof}

\section{Proofs}\label{sec:proof}

\subsection{Proof of Theorem \ref{car_sr}}

Let 
\[\V = \bigcup_O \comp(O)\] be the set of stable configurations on $G$ compatible with some orientation. We wish to show that $\s(G) = \V$. We do this by showing that they are subsets of each other.

\begin{lem} Let  $G = (V \cup \{s\},E) \in {\cal G}$. We have
\[\s(G) \subseteq \V.\]
\end{lem}

\begin{proof}
Firstly we show that $\eta^{max} \in \V$. For this, consider a spanning tree of $G$ and orient all edges of this tree outwards from the sink; that is, if $\{a,b\}$ is an edge of the spanning tree such that $a$ is closer to the sink than $b$, orient the edge $a \rightarrow b$. Orient all remaining edges in any direction. Then for the resulting orientation every non-sink vertex has at least one incoming edge. Condition \eref{comp_cond} now shows that $\eta^{max}$ is compatible with this orientation, so $\eta^{max} \in \V$.

Now let $O$ be an orientation compatible with $\eta^{max}$, and let $\eta' \in \s(G)$. We shall construct an orientation $O'$ which is compatible with $\eta'$. To do this, consider a history of grain additions and stochastic topplings that leads from $\eta^{max}$ to $\eta'$. We construct $O'$ iteratively from $O$ by making the following changes to the orientation at each step of this history:
\begin{itemize}
\item If a grain topples from a vertex $a$ to a neighbour $b$ and the edge $\{a,b\}$ is oriented $a \rightarrow b$, we reverse the orientation of this edge so that it is oriented $b \rightarrow a$.
\item Otherwise --- that is, if the edge $\{a,b\}$ is oriented $b \rightarrow a$, or if a grain is added --- do nothing.
\end{itemize}
These changes are shown in Figure \ref{fig:history}.
\begin{figure}\centering
\psfrag{e}{$e$}
\includegraphics[height=2.6cm]{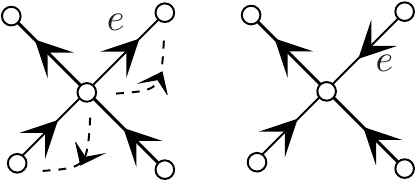}
\captionn{After the vertex topples according to the dashed arrows, only the edge $e$ is reversed.}\label{fig:history}
\end{figure}

We now show that $\eta' \in\comp(O')$. Let $\eta_1=\eta^{max},\eta_2,\ldots,\eta_k=\eta'$ be the configurations constructed at each step of the history, some of which may be unstable. Let $O_1=O,...,O_k=O'$ be the corresponding orientations. We show by induction on $i$ that for any $i \in \lbrace 1, \ldots, k \rbrace$, $\eta_i$ is compatible with $O_i$. For $i=1$ we have $\eta^{max} \in \comp(O)$ by definition. At any fixed $i > 1$, there are two cases:
\begin{enumerate}
\item $\eta_i$ is reached from $\eta_{i-1}$ by addition of a grain at some vertex $a$. 

For any vertex $v$, we have $l_{\eta_i}(v) \leq l_{\eta_{i-1}}(v)$ and by construction $\In_{O_i}(v) = \In_{O_{i-1}}(v)$. Since $\eta_{i-1}$ is compatible with $O_{i-1}$, condition \eref{comp_cond} is satisfied for $\eta_i$ and $O_i$ for all vertices in $V$. Thus $\eta_i$ is compatible with $O_i$.
\item $\eta_i$ is reached from $\eta_{i-1}$ through a (legal) toppling of some vertex $a$.

Let the neighbours of $a$ be $b_1,\ldots,b_m$, and suppose the toppling transfers $\alpha_1,\ldots,\alpha_m$ grains from $a$ to each of these vertices respectively. Since the toppling is legal, $\eta_{i-1} |_a \geq d(a) + 1$, and so after the toppling, $l_{\eta_i}(a) \leq \sum_j \alpha_j - 1$. But now grains have toppled out from $a$ along $\sum_j \alpha_j$ edges, so by construction all these edges are oriented towards $a$. Therefore $\In_{O_i}(a) \geq \sum_j \alpha_j$ and condition \eref{comp_cond} is satisfied at $a$.

It remains to check that \eref{comp_cond} is still satisfied at each $b_j$. At most $\alpha_j$ incoming edges to $b_j$ in $O_{i-1}$ are now outgoing in $O_i$, so $\In_{O_i}(b_j) \geq \In_{O_{i-1}}(b_j) - \alpha_j$. Furthermore, $l_{\eta_i}(b_j) = l_{\eta_{i-1}}(b_j) - \alpha_j$, so by induction \eref{comp_cond} is satisfied at $b_j$.

No other vertices apart from $a$ and its neighbours are changed, so $\eta_i$ is compatible with $O_i$.
\end{enumerate}

Taking $i = k$, we have shown that $O'$ is compatible with $\eta'$. This completes the proof of the lemma.
\end{proof}

\begin{lem} Let  $G = (V \cup \{s\},E) \in {\cal G}$. We have
\[\V \subseteq \s(G).\]
\end{lem}
\begin{proof}
We use induction on the number of vertices $|V|$. If $|V| = 1$, then $G$ is simply one vertex $v$, connected to $s$ by say $k$ edges. Then $\V$ consists of all configurations $\eta$ where $l_\eta(v) \leq k-1$, which are SR by Proposition \ref{pro:init_cond}.

Suppose now that the lemma is true for any graph with $|V| = k-1$. Take a graph $G=(V \cup \lbrace s \rbrace,E)$ with $\vert V \vert = k$ and a configuration $\eta \in \V$. We wish to show that $\eta$ is SR. To do this, take an orientation $O$ of $G$ such that $\eta \in \comp(O)$, and choose a vertex $u$ connected to $s$. We can assume that all edges connecting $u$ to $s$ are oriented towards $u$ (if some are not, we can reverse them and the resulting orientation remains compatible with $\eta$). Write $m_u$ for the number of such edges. 

Let $G_u$ be the graph obtained from $G$ by identifying $u$ with $s$, removing all $\{u,s\}$ edges (denote the new sink by $s.u$). Note in particular that the degree of any vertex in $G_u$ (except $s.u$) is equal to its degree in $G$. We now let $O_u$ be the orientation on $G_u$ coinciding with $O$ on all edges still present, except that any edge connected to $s.u$ is oriented away from it.

For any vertex $v$ in $G_u$, let  $k_v$ be the number of edges oriented $v\to u$ in $O$ ($k_v\geq 0$). 
We define a configuration $\eta^u$ on $G_u$ by $\eta^u_v = \eta_v - k_v$ for $v\neq s.u$ in $G_u$. Now for any such $v$, we have $\In^{G_u}_{O_u}(v) = \In^G_O(v) + k_v$ and $l^{G_u}_{\eta^u}(v) = l^G_{\eta}(v) + k_v$. Since $\eta \in \comp(O)$, we deduce that $\eta^u \in \comp(O_u)$. Hence by induction $\eta_u \in \s(G_u)$, and therefore there exists a history of grain additions and legal  topplings that leads from $\eta^{max}$ to $\eta^u$ on $G_u$.

We now start from $\eta^{max}$ on $G$ and copy this history (all legal topplings remain legal because the degrees in the two graphs are identical). This results in a configuration $\eta'$ on $G$ with $\eta'_v = \eta^u_v$ if $v \neq u$. We then continue the history by either adding grains to $u$ or repeatedly toppling grains from $u$ to $s$ until $u$ has $d^G(u) + 1$ grains, i.e. it is minimally unstable.

We now make one final toppling at $u$, sending $k_v$ grains to each of its neighbours $v$, and $\max(l^G_{\eta}(u) - \sum_v k_v + 1, 0)$ grains to the sink. In order to do this, we must have at least $l^G_{\eta}(u) - \sum_v k_v + 1$ edges $\{u,s\}$. However, since $\eta$ and $O$ are compatible, we know that $l^G_{\eta}(u) + 1 \leq m_u + \sum_v k_v$, so this is true. Denote by $\eta''$ the configuration we finally reach.

We have:
\begin{itemize}
\item If $v$ is neither $u$ nor one of its neighbours, then $l^G_{\eta''}(v) = l^G_{\eta'}(v) = l^{G_u}_{\eta^u}(v) = l^G_{\eta}(v)$.
\item If $v$ is a neighbour of $u$, then $l^G_{\eta''}(v) = l^G_{\eta'}(v) - k_v = l^{G_u}_{\eta^u}(v) - k_v = l^G_{\eta}(v)$.
\item $l^G_{\eta''}(u) \geq \sum_v k_v + l^G_{\eta}(u) - \sum_v k_v + 1 - 1 = l^G_{\eta}(u)$.
\end{itemize}
Together, this shows that $\eta''$ is identical to $\eta$ except at $u$, where it may have less grains. We then merely add the difference in grains to $u$, and have thus created a history of grain additions and legal topplings which leads from $\eta^{max}$ to $\eta$. Therefore $\eta$ is SR and the lemma, and Theorem \ref{car_sr}, is proved.
\end{proof}

\subsection{Proof of Theorem \ref{th:car_dr}}
Let $\eta \in \Stable(G)$ be compatible with an orientation with no directed cycles. Let $O$ be such an orientation. Since $O$ has no directed cycles, and we can take all
edges adjacent to the sink as oriented away from it, we can order the
vertices of $G$ as $s = v_0,v_1,\ldots,v_n$ so that there exist no edges
$v_j \rightarrow v_i$ for $i < j$.

Now we apply the burning algorithm to $\eta$. We claim by induction that
this can burn the vertices in the order described above. Since $s = v_0$, the
initial condition is trivial. Now suppose we have burned vertices
$v_0,\ldots,v_{i-1}$. All incoming edges to $v_i$ are burnt, so the
number of unburned edges adjacent to $v_i$ is $d(v_i) - \In_O(v_i)$. But by
\eref{comp_cond},
\[\eta_{v_i} = d(v_i) - l_\eta(v_i) \geq d(v_i) - \In_O(v_i) + 1.\]
Therefore $v_i$ can be burnt. Thus the burning algorithm burns all the
vertices of the graph, and $\eta$ is deterministically recurrent.

Conversely, let $\eta \in \D(G)$. Now apply the burning algorithm
to $\eta$, and every time we burn a vertex $v$, orient all edges from
previously burnt vertices to $v$ as incoming edges to $v$. Since all
vertices are burnt, this produces a full orientation $O$ on $G$, which
obviously has no directed cycles. From the burning condition, we know that for all $v$, $\eta_v$ is greater
than the number of unburnt edges, which is $d(v) - \In_O(v)$. This gives
\[l_\eta(v) = d(v) - \eta_v < d(v) - \left( d(v) - \In_O(v) \right) =
\In_O(v).\]
Since these are integers, this means that \eref{comp_cond} is fulfilled for
all vertices. Thus $O$ is compatible with $\eta$ and the theorem is proved. $~\Box$

\subsection{Proof of Lemma \ref{pro_TB}}

We construct a bijection $\Phi$ from $\s(G \setminus T)$ to $\s(G)$ such that for any $\eta \in \s(G \setminus T)$, $l(\eta) = l(\Phi(\eta))$. Define for $\eta \in \s(G \setminus T)$ and $v \in V$
\[ l^G_{\Phi(\eta)}(v) = \left\lbrace \begin{array}{rl}
0 & \mbox{ if } v\in V', \\
l^{G \setminus T}_\eta(v) & \mbox{ otherwise.}
\end{array} \right. \]

Take an orientation $O$ on $G \setminus T$ which is compatible with $\eta$, and extend this orientation to $G$ by orienting each edge in $T$ away from $r$. Then $\Phi(\eta)$ is compatible with the resulting orientation, so $\Phi(\eta) \in \s(G)$. Moreover, $\Phi$ is clearly an injection. 

It remains to show that it is surjective. To see this, consider a configuration $\eta \in \s(G)$ and a compatible orientation $O$. Each vertex in $V'$ must have at least one incoming edge in $E'$. But since $T$ is a tree, $|V'| + 1 = |E'| + 1$. Thus each edge in $E'$ points to a different vertex in $T$, so $\In^G_O(v) = 1$ for all $v \in V'$. This implies that $l^G_\eta(v) = 0$. Furthermore, all edges in $E'$ adjacent to $r$ point away from it.

Now define $\eta'$ on $G \setminus T$ according to $l^{G \setminus T}_{\eta'}(v) = l^G_\eta(v)$, and let $O'$ be the restriction of $O$ to $G \setminus T$. We have $\In^{G \setminus T}_{O'}(v) = \In^G_O(v)$ for all $v \in G \setminus T$, so clearly $\eta'$ is compatible with $O'$ and $\Phi(\eta') = \eta$. Thus $\Phi$ is a bijection.

\subsection{Proof of Theorem \ref{thm:LP}}

Let $G = (V \cup \lbrace s \rbrace,E) \in {\cal G}$, and let $e$ be an edge of $E$ which is neither a bridge nor connected to the sink. Write $e = \{a,b\}$. For $\eta \in \s(G)$ we distinguish the following two cases:
\begin{enumerate}[(A)]
\item There exists an orientation $O$ on $G$, compatible with $\eta$, such that $e$ is oriented $a \rightarrow b$ in $O$, and $l^G_{\eta}(b) > 0$.
\item For all orientations $O$ compatible with $\eta$, all $\{a,b\}$ edges are oriented $b \rightarrow a$, or $l^G_{\eta}(b) = 0$.
\end{enumerate}
We write $\eta \in \s_{(x)}(G)$ if $\eta$ satisfies condition $x \in \lbrace A,B \rbrace$. Obviously $\s(G) = \s_{(A)}(G) \cup \s_{(B)}(G)$.

Now we define a function $f: \s(G) \rightarrow \Stable(G \setminus e) \cup \Stable(G.e)$ as follows:
\begin{itemize}
\item If $\eta \in \s_{(A)}(G)$ then $f(\eta) = \eta - \1_a$  is a configuration on $G \setminus e$.
\item If $\eta \in \s_{(B)}(G)$ then $f(\eta)|_v = \left\lbrace \begin{array}{rl}
\eta_a + \eta_b - 2 & \mbox{ if } v=a.b, \\
\eta_v & \mbox{ otherwise,}
\end{array}
\right. $ is a configuration on $G.e$.
\end{itemize}
To simplify further calculations, we note that in the first case, $l^{G\setminus e}_{f(\eta)}(a) = l^G_\eta(a)$ and $l^{G\setminus e}_{f(\eta)}(b) = l^G_\eta(b) - 1$. In the second case, $l^{G.e}_{f(\eta)}(a.b) = l^G_\eta(a) + l^G_\eta(b)$.

It is easy to see that if $\eta$ is stable, $f(\eta)$ is also stable --- the lacking number can decrease by at most one, and this occurs only at $b$ when $\eta \in \s_{(A)}(G)$, where by definition $l^G_\eta(b) > 0$. Moreover, we have
\[l(f(\eta)) = \left\{ \begin{array}{rl} l(\eta) - 1 & \mbox{ if $\eta \in \s_{(A)}(G)$,} \\ l(\eta) & \mbox{ if $\eta \in \s_{(B)}(G)$.} \end{array} \right.\]
In light of this, it is sufficient to show the following theorem to prove Theorem \ref{thm:LP} since the two sets $\s(G \setminus e)$ and $\s(G.e)$ are disjoint, being configurations on different graphs.

\begin{thm}\label{thm:bij}  Let $G = (V \cup \lbrace s \rbrace,E) \in {\cal G}$.
The function $f$ defined above is a bijection from $\s(G)$ to $\s(G \setminus e) \cup \s(G.e)$.
\end{thm}

This theorem is itself a direct consequence of the four following lemmas. The first lemma shows that for any $\eta \in \s(G)$, the configuration $f(\eta)$ is indeed stochastically recurrent.

\begin{lem}\label{well_def} For any $G = (V \cup \lbrace s \rbrace,E)$ in ${\cal G}$,
\[f(\s(G)) \subseteq \s(G \setminus e) \cup \s(G.e).\]
\end{lem}

\begin{proof}
Take $\eta \in \s(G)$, and fix an orientation $O$ on $G$ such that $\eta \in \comp(O)$. There are two cases.
\begin{enumerate}
\item $\eta \in \s_{(A)}(G)$.

By construction, we may choose $O$ such that $e$ is oriented $a \rightarrow b$. Let $O'$ be the orientation on $G\setminus e$ which is identical to $O$ on all edges of $G \setminus e$. Then for any vertex $v$,
\[ \In^{G \setminus e}_{O'}(v) = \In^G_O(v) - \delta_{v,b} \geq 1 + l^G_\eta(v) - \delta_{v,b} = 1 + l^{G \setminus e}_{f(\eta)}(v), \]
so $f(\eta) \in \comp(O')$. Thus, by Theorem \ref{car_sr}, $f(\eta) \in \s(G \setminus e)$.
\item $\eta \in \s_{(B)}(G)$.

Let $k$ be the number of $\{a,b\}$ edges in $G$. Then in $G.e$, these will be replaced by $k-1$ edges $\{a.b,s\}$. Orient these as $s \rightarrow a.b$, and orient all other edges of $G.e$ as they are oriented in $O$. Denote by $O'$ the resulting orientation on $G.e$. Then
\[ \In^{G.e}_{O'}(a.b) = \In^G_O(a) + \In^G_O(b) - 1 \geq (l^G_{\eta}(a) + 1) + (l^G_{\eta}(b) + 1) - 1 = l^{G.e}_{f(\eta)}(a.b) + 1, \]
so condition \eref{comp_cond} is satisfied at $a.b$. Since $\In^{G.e}_{O'}(v) = \In^G_O(v)$ for $v \neq a.b$, it is clearly also satisfied elsewhere, so $f(\eta) \in \comp(O')$, and by Theorem \ref{car_sr}, $f(\eta) \in \s(G.e)$.
\end{enumerate}
\end{proof}

We write $f_{(A)}$ (resp. $f_{(B)}$) for the restriction of $f$ to the set $\s_{(A)}(G)$ (resp. $\s_{(B)}(G)$). We will show that each of these are bijections onto their respective images.

\begin{lem}\label{bij_a}For any $G = (V \cup \lbrace s \rbrace,E)$ in ${\cal G}$, 
the function $f_{(A)}$ is a bijection from $\s_{(A)}(G)$ to $\s( G \setminus e)$.
\end{lem}

\begin{proof}
The fact that $f_{(A)}$ is injective follows immediately from the definition of $f$. To show that $f_{(A)}$ is surjective onto $\s( G \setminus e)$, let $\eta \in \s( G \setminus e)$ and take an orientation $O$ on $G \setminus e$ compatible with $\eta$. Let $O' = O \cup \{ a \rightarrow b \}$, and define $\eta' = \eta + \1_a$ as a configuration on $G$. It is obvious that $\eta'$ is stable and $f(\eta') = \eta$. Then for any vertex $v$,
\[\In^G_{O'}(v) = \In^{G \setminus e}_O(v) + \delta_{v,b} \geq 1 + l^{G \setminus e}_O(v) + \delta_{v,b} = 1 + l^G_{\eta'}(v). \]
Thus $\eta' \in \comp(O')$ so $\eta'$ is SR. Moreover, $e$ is oriented $a \rightarrow b$ in $O'$ and $l^G_{\eta'}(b) > 0$, so $\eta' \in \s_{(A)}(G)$ as desired.
\end{proof}

\begin{lem}\label{surj_b} For any $G = (V \cup \lbrace s \rbrace,E)$ in ${\cal G}$, the function $f_{(B)}$ is surjective onto $\s(G.e)$.
\end{lem}
\begin{proof}
For any $\eta \in \s(G.e)$, let $O$ be a compatible orientation on $G.e$. Let $k$ be the multiplicity of the edge $e$ in $G$ (as before we may have $k=1$). We may assume that the $k-1$ edges $\{s,a.b\}$ in $G.e$ are all oriented $s \rightarrow a.b$ in $O$. Write $i_a$ (resp. $i_b$) for the number of edges oriented into $a.b$ in $O$ which correspond to edges into $a$ (resp. $b$) in $G$, from vertices other than $b$ (resp. $a$). This is illustrated in Figure \ref{fig:edges}. 
\begin{figure}\centering 
\psfrag{a}{$a$}\psfrag{b}{$b$}\psfrag{k-1}{$k-1$}\psfrag{k}{$k$}\psfrag{i}{$i$}
\psfrag{i_a}{$i_a$}\psfrag{i_b}{$i_b$}
\includegraphics[width=13cm]{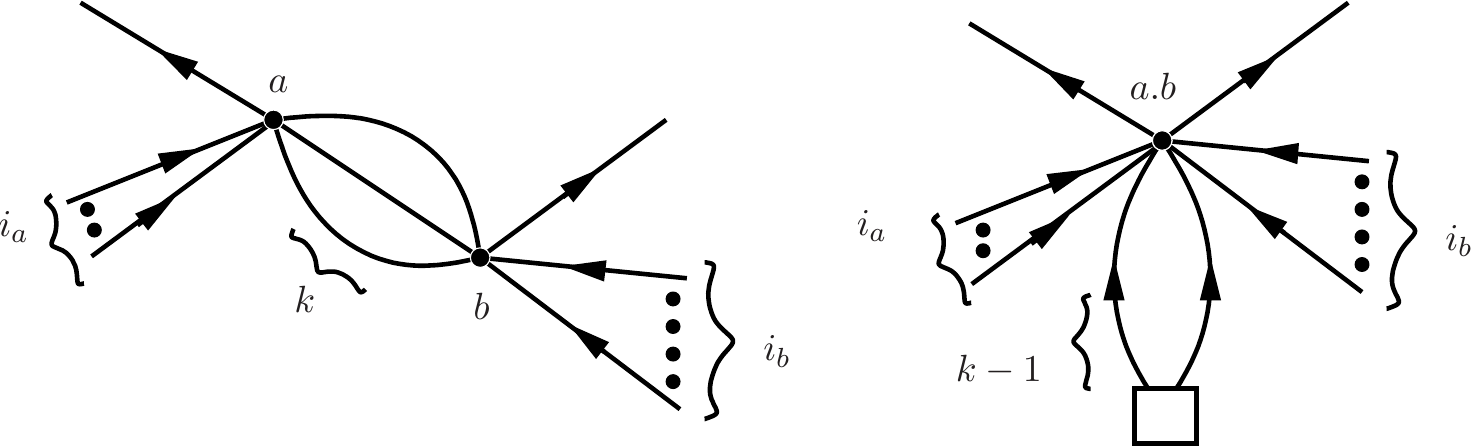}
\captionn{Counting the number of edges in $O'$ (left) and $O$ (right). The $\{a,b\}$ edges are oriented according to the value of $i_a$ and $l_{\eta}^{G.e}(a.b)$.\label{fig:edges}}
\end{figure}

Since $\eta$ is compatible with $O$, we have 
\beq\label{cond_a.b}
 \In^{G.e}_O(a.b) = i_a + i_b + k-1 \geq 1 + l^{G.e}_{\eta}(a.b).
\eq

Let $\bar{f}_{(B)} : \s(G) \rightarrow \Stable(G.e)$ be the extension of $f_{(B)}$ to all of $\s(G)$, i.e.
\[\bar{f}_{(B)}(\eta)|_v = \left\{ \begin{array}{rl}
\eta_a + \eta_b - 2 & \mbox{ if } v=a.b, \\
\eta_v & \mbox{ otherwise.}
\end{array} \right. \]
The difference between $\bar{f}_{(B)}$ and $f_{(B)}$ lies in the fact that $\bar{f}_{(B)}$ is defined over all $\s(G)$, whereas $f_{(B)}$ is defined only on $\s_{(B)}(G)$. We will now define a configuration $\eta' \in \s(G)$ such that $\bar{f}_{(B)}(\eta') = \eta$. Firstly let $\eta'_v = \eta_v$ if $v \neq a,b,a.b$. Likewise let $O'$ be an orientation on $G$ where all edges in $G.e$ are oriented identically to $O$ (the remaining edges are as yet unspecified). Obviously \eref{comp_cond} is satisfied for $\eta'$ and $O'$ at vertices other than $a$ and $b$. We now assign grains to $a$ and $b$, and orientations to the $\{a,b\}$ edges, according to 3 cases.

\begin{enumerate}
\item $i_a=0$.

We set $l^G_{\eta'}(a) = 0$, $l^G_{\eta'}(b) = l^{G.e}_{\eta}(a.b)$. In $O'$, orient one $\{a,b\}$ edge as $b \rightarrow a$ and the remaining as $a \rightarrow b$. Clearly, \eref{comp_cond} is satisfied at $a$ (which has one incoming edge), and
\[ \In^G_{O'}(b) = i_b + k-1 = i_a + i_b + k-1 \geq 1 + l^{G.e}_{\eta}(a.b) = 1 + l^G_{\eta'}(b), \]
from \eref{cond_a.b}, so \eref{comp_cond} is also satisfied at $b$.

\item $1 \leq i_a \leq 1 + l^{G.e}_{\eta}(a.b)$.

We set $l^G_{\eta'}(a) = i_a - 1 \geq 0$, $l^G_{\eta'}(b) = l^{G.e}_{\eta}(a.b) - l^G_{\eta'}(a) \geq 0$. In $O'$, orient all $k$ edges $\{a,b\}$ as $a \rightarrow b$. Then
\[ \In^G_{O'}(a) = i_a = 1 + l^G_{\eta'}(a),\] 
\[ \In^G_{O'}(b) = i_b + k \geq 1 + l^{G.e}_{\eta}(a.b) - i_a + 1 = 1 + l^G_{\eta'}(b),\]
again using \eref{cond_a.b}, so \eref{comp_cond} is satisfied at both $a$ and $b$.

\item $i_a > 1 + l^{G.e}_{\eta}(a.b)$.

We set $l^G_{\eta'}(a) = l^{G.e}_{\eta}(a.b) $, $l^G_{\eta'}(b) = 0$. In $O'$, orient all $k$ edges $\{a,b\}$ as $a \rightarrow b$. Then we have $ \In^G_{O'}(a) = i_a \geq 1 + l^G_{\eta'}(a)$, and $ \In^G_{O'}(b) = i_b + k \geq 1 = 1 + l^G_{\eta'}(b)$. Thus condition \eref{comp_cond} is satisfied at $a$ and $b$.
\end{enumerate}

Now, in each of these cases we have $l^{G.e}_{\eta}(a.b) = l^G_{\eta'}(a) + l^G_{\eta'}(b)$, so $\bar{f}_{(B)}(\eta') = \eta$. Likewise, $\eta'$ is compatible with $O'$, so $\eta' \in \s(G)$. It remains to show that we may choose an $\eta' \in \s_{(B)}(G)$ so that $f(\eta') = \eta$.

To show this, define $l_b^{min} = \min \{l^G_{\eta'}(b): \eta' \in \s(G) \mbox{ s.t. } \bar{f}_{(B)}(\eta') = \eta\}$. Since there exists at least one such $\eta'$, this is well-defined. Now take $\eta' \in \s(G)$ such that $l^G_{\eta'}(b) = l_b^{min}$ and $\bar{f}_{(B)}(\eta') = \eta$. We show that $\eta' \in \s_{(B)}(G)$.

If $l_b^{min} = 0$, this is true by definition. Now suppose that $l_b^{min} > 0$ and there exists an orientation $O'$ on $G$ compatible with $\eta'$ with an edge oriented $a \rightarrow b$. We define the orientation $O''$ as $O'$ with that edge reversed and all other edges oriented as in $O$. Likewise, define the configuration $\eta'' = \eta' + \1_b - \1_a$, so that $l^G_{\eta''}(a) = l^G_{\eta'}(a) + 1$, $l^G_{\eta''}(b) = l^G_{\eta'}(b) - 1 \geq 0$, and $l^G_{\eta''}(v) = l^G_{\eta'}(v)$ elsewhere. Since $\eta' \in \comp(O')$, we have $\eta'' \in \comp(O'')$ by construction. Now we have $\bar{f}_{(B)}(\eta'') = \eta$ and $\eta'' \in \s(G)$, but $l^G_{\eta''}(b) < l^{min}_b$. This is a contradiction of the definition of $l^{min}_b$. Therefore no such orientation $O'$ exists, and $\eta' \in \s_{(B)}(G)$.
\end{proof}

\begin{lem}\label{inj_b}For any $G = (V \cup \lbrace s \rbrace,E)$ in ${\cal G}$, 
the function $f_{(B)}$ is injective.
\end{lem}

\begin{proof}
Let $\eta_1, \eta_2 \in \s_{(B)}(G)$ such that $f(\eta_1) = f(\eta_2)$, that is $l^G_{\eta_1}(a) + l^G_{\eta_1}(b) = l^G_{\eta_2}(a) + l^G_{\eta_2}(b)$ and $l^G_{\eta_1}(v) = l^G_{\eta_2}(v)$ if $v\neq a,b$, but suppose $\eta_1 \neq \eta_2$. Assume without loss of generality that $l^G_{\eta_2}(b) > l^G_{\eta_1}(b) \geq 0$. Now choose compatible orientations $O_1$ and $O_2$ respectively on $G$. Suppose that the edge $e$ has multiplicity $k \geq 1$ in $G$. Since $\eta_2 \in \s_{(B)}(G)$ and $l^G_{\eta_2}(b) > 0$, these $k$ edges must be oriented $b \rightarrow a$ in $O_2$.

Now, if $\In^G_{O_2}(a) > 1 + l^G_{\eta_2}(a)$, then reversing the orientation of $e$ to $a \rightarrow b$ in $O_2$ results in another orientation compatible with $\eta_2$, contradicting the fact that $\eta_2 \in \s_{(B)}(G)$. Therefore these quantities are equal and
\beq\label{eq:lem47} \In^G_{O_1}(a) \geq 1 + l^G_{\eta_1}(a) > 1 + l^G_{\eta_2}(a) = \In^G_{O_2}(a). \eq
Now let $\Delta$ be the set of edges in $O_2$ which are oriented differently from $O_1$. We define a subgraph $T = (V(T),E(T))$ of $G$ as the union of all directed paths in $O_2$ starting from $a$ whose edges are in $\Delta$ (and the induced vertices). From (\ref{eq:lem47}), this contains at least one edge adjacent to $a$.

Firstly, we claim that $b \notin V(T)$. Otherwise, there exists a directed path from $a$ to $b$ in $O_2$. Starting from $O_2$, we may reverse $e$ and all edges of this path to reach an orientation with the same number of incoming edges at each vertex as $O_2$, and therefore compatible with $\eta_2$, but with $e$ oriented $a \rightarrow b$. This contradicts the assumption that $\eta_2 \in \s_{(B)}(G)$.

Now start from $O_2$ and reverse the orientation of $e$ and all edges in $E(T)$. Denote this orientation by $O'_2$. We show that $O'_2$ is compatible with $\eta_2$ by checking condition \eref{comp_cond} at $b$ and vertices in $V(T)$ (which include $a$):
\begin{itemize}
\item $\In^G_{O_2'}(b) = \In^G_{O_2}(b) + 1 \geq 1 + l^G_{\eta_2}(b)$, since $e$ has been reversed and $b \notin V(T)$.
\item For $v \in V(T)$, all incoming edges to $v$ in $O_1$ are identically oriented in $O'_2$ by construction. Therefore $\In^G_{O_2'}(v) \geq \In^G_{O_1}(v) \geq 1 + l^G_{\eta_1}(v) \geq 1 + l^G_{\eta_2}(v)$, where the last inequality is strict if $v = a$ and an equality otherwise.
\end{itemize}
This gives us an orientation compatible with $\eta_2$ containing an edge $a \rightarrow b$. Again, this is a contradiction of the assumption that $\eta_2 \in \s_{(B)}(G)$. Thus there cannot exist configurations $\eta_1 \neq \eta_2$ in $\s_{(B)}(G)$ such that $f(\eta_1) = f(\eta_2)$, and $f_{(B)}$ is injective.
\end{proof}

\section{Conclusion}
\label{sec:conclusion}

In this paper, we have devised a generalisation of the ASM in which the topplings are stochastic. This model behaves qualitatively differently to the established ASM of Dhar. In particular, the set of recurrent states of this model contains that of the former model. We have proved a characterisation of these states using graph orientations. We also define a generating function of these states which counts the number of ``lacking" grains, and show that this ``lacking polynomial" satisfies a recurrence relation which resembles that of the Tutte polynomial.

There are two directions in which to advance this work. Given the many combinatorial interpretations of the Tutte polynomial, it would be of interest to see if the lacking polynomial demonstrates similar interpretations. In other words, the lacking polynomial may count certain combinatorial objects for given values of its parameter, and we would like to determine what these objects are. Alternatively, the lacking polynomial may be related in some way to the Tutte polynomial, and the nature of this relation should be determined precisely.

The other topic of interest is to probe further into the behaviour of the Markov chain structure, more specifically the steady state. In the ASM, all recurrent states are equally likely, but this is not the case for the SSM. It would be interesting to calculate the probabilities for the stochastically recurrent states. Once we have done so, we can analyse the behaviour of the model in the steady state, and see if it displays a similar power-law behaviour to that observed for the classic model.
\small
\renewcommand{\baselinestretch}{1}

\bibliographystyle{abbrv}

%\bibliography{SP1}

\end{document}